\documentclass[11pt]{amsart}

\usepackage{latexsym,enumerate}

\usepackage{amsmath,amssymb,amsthm,amsfonts,latexsym}
\usepackage[bookmarks]{hyperref}
\hypersetup{backref, colorlinks=true}

\let\=\partial

\newtheorem {pro}{Proposition}[subsection]
\newtheorem {thm}[pro]{Theorem}
\newtheorem {cor}[pro]{Corollary}
\newtheorem{lem}[pro]{Lemma}

\theoremstyle{definition}
 \newtheorem {rem}[pro]{Remark}
 
 \newtheorem {rems}[pro]{Remarks}
\newtheorem {dfn}[pro]{Definition}

\newcommand{\Q} {\mathbb{Q}}

\newcommand{\ko}{ k(0_+)}

\newcommand{\R} {\mathbb{R}}
\newcommand{\Z} {\mathbb{Z}}
\newcommand{\N} {\mathbb{N}}
\newcommand{\C} {\mathcal{C}}

\newcommand{\St}{\mathcal{S}}
\newcommand{\xb}{X_{reg}}

\newcommand{\D}{\mathcal{D}}

\newcommand{\E}{\mathcal{E}}
\newcommand{\Pp}{\mathbb{P}}
\newcommand{\la}{\mathcal{L}}
\newcommand{\eqr}{\sim}

\newcommand{\ep}{\varepsilon}

\title {{\bf   $L^\infty$ cohomology is intersection cohomology}}



\author{Guillaume Valette}

\address{Fields Institute 222 College Street,  Toronto, Ontario  M5T 3J1, Canada}
\address{
Instytut Matematyczny PAN, ul. Sw. Tomasza 30, 31-027 Krak\'ow, Poland }

\email{gvalette@impan.pl}


\thanks{}


\begin{document}
\maketitle

\begin{abstract}
Let $X$ be a subanalytic compact pseudomanifold. We show a de
Rham theorem for $L^\infty$ forms on the nonsingular part of $X$. We prove that their cohomology  is isomorphic to the intersection cohomology of $X$ in the
maximal perversity.
\end{abstract}

\section{Introduction}During the three last decades, many authors studied  $L^p$ differential forms on singular varieties.
The history started with J.  Cheeger  who computed the cohomology of $L^2$
forms on pseudomanifolds with metrically conical singularities \cite{c1}.
 He
proved in \cite{c2} that the $L^2$ cohomology is actually isomorphic (for pseudomanifolds with metrically conical singularities) to   intersection homology in the middle perversity (see also \cite{cgm}).

 Intersection homology  was introduced independently by M. Goresky and R. MacPherson in \cite{gm} in order  to study the topology of singular sets. Its main feature is to satisfy  Poincar\'e duality for a large class of
singularities, sufficiently  general to  enclose all the complex projective analytic  varieties (see \cite{gm, gm2}).

 Cheeger's de Rham theorem thus
provided a means to investigate the topology of singular sets via differential
geometry. It also enabled to carry out a Hodge theory on pseudomanifolds with metrically conical singularities, which was developed by Cheeger himself in a series of works \cite{c1,c2,c3,c4}.

 $L^p$ cohomology, $p\neq 2$, turned out to
be related to intersection homology as well.  Let us mention some
of the many related works which then  appeared.   In \cite{y}, Y.
Youssin computes the $L^p$ cohomology groups   of spaces with
conical horns. He shows that the $L^p$ cohomology groups are
isomorphic to intersection cohomology groups in the so-called
$L^p$ perversity $1<p<\infty$. He also describes quite explicitly
the case of $f$-horns. The so-called $f$-horns are
cones endowed with a metric decreasing at a rate proportional to a
function $f$ of the distance to the origin. L. Saper studies  in
\cite{s2} the $L^2$ cohomology  for sets with isolated
singularities with a distinguished  K\"ahler metric.

In \cite{hp}, the authors focus on normal algebraic complex
surfaces (not necessarily metrically conical).  They also show
that the $L^2$ cohomology is dual to intersection homology (see
also \cite{s1}).

 In
\cite{bgm}, the authors show, on a simplicial complex, an explicit
isomorphism between the $L^p$ shadow forms and intersection
homology. The shadow forms are smooth forms constructed by the
authors in a  combinatorial way, like Whitney forms \cite{wh}.

It is striking that, all the above mentioned de Rham theorems
include an assumption on the metric type of the singularities or
are devoted to low dimensional singular sets   whose metric type  is easier to handle.
In this paper, we focus on
$L^\infty$ forms, i.e., forms having a bounded size. We prove a de
Rham theorem for  any compact subanalytic pseudomanifold,
establishing an isomorphism between $L^\infty$ cohomology and
intersection homology in the maximal perversity (Theorem \ref{thm de rham}). We also prove
that the isomorphism is provided by integration on subanalytic
singular chains. The class of subanalytic pseudomanifolds covers a large class of subsets such as all the complex analytic projective varieties. Furthermore, the theory presented in this paper could go over singular subanalytic subsets which are not pseudomanifolds and we could adapt the statement to arbitrary subanalytic subsets.

This theorem, which applies to any compact subanalytic pseudomanifold, is
proved by looking in details at the metric structure of subanalytic
sets (see section $2$). The sharp description of the metric type of singularities
obtained \cite{v1,v2} will make it possible to work without any extra
assumption on the metric type
of the singularities.

As a consequence, we immediately see the $L^\infty$ groups are finitely generated.
The purpose is also, as in the case of $L^2$ cohomology, to
find a category of forms for which we can carry out a Hodge theory
 for any compact subanalytic singular  variety.   Performing analysis or differential geometry on singular spaces is much more challenging that on smooth manifolds because the metric geometry of singular sets is much harder to handle.



\subsection*{Acknowledgment.} The author is very happy to thank
Pierre Milman for his encouragements  and valuable discussions on
this topic.

\section{Definitions and the main result}
 This paper deals with  subanalytic sets. We recall their definition and outline their basic properties in an Appendix at the end of the paper (sections $5$ and $6$).

\subsection{$L^\infty$-cohomology groups.} Before stating the main result, we need to define the $L^\infty$ cohomology groups.

\begin{dfn} Let $Y $ be  a  $C^\infty$ submanifold of $\R^n$.
 As $Y$ is embedded in $\R^n$, it inherits  a natural structure of Riemannian manifold. We say that a $j$-differential form
 $\omega$ on $Y$ is $L^\infty$ if there exists a constant $C$ such that  for any $x \in Y$: $$|\omega(x)| \leq C, $$
where $|\omega(x)|$ denotes the norm of $\omega(x)$ (as a linear
mapping).  We will write  $d$ for the exterior differential operator.

 We denote by $\Omega^j_\infty (Y)$ the real  vector space
constituted by all the differential $C^\infty$ $j$-forms $\omega$ such that
$\omega $ and $d\omega$ are both $L^\infty$.

 Given $\omega\in \Omega_{\infty}^j(Y)$, we set $|\omega|_\infty:=\sup_{x \in Y}
|\omega(x)|$.

The cohomology groups of this cochain complex are called the {\bf
$L^\infty$ cohomology groups of $Y$} and will be denoted by
$H^\bullet _\infty(Y)$.
\end{dfn}

\subsection{Intersection homology in the maximal perversity and the main theorem.} We recall below the definition of intersection homology (see \cite{gm}). Intersection homology as defined in the latter article  depends on a
"perversity". The  definition below corresponds to the case of the maximal perversity $t=(0,1,\dots,l-2)$ (the letter $t$ stands for "top" perversity).

 As we will be interested in the only case of the maximal perversity, we  specify this particular case in the definition and shall not introduce the technical notion of perversity. But this  accounts for the notation $I^t C_j(X)$  (which is the notation of \cite{gm}) used below.


Given a subanalytic  set $X$, we denote by $X_{reg}$ the set of points of $X$
at which $X$ is locally a $C^\infty$ manifold (without boundary, of any dimension) and we will
write  $X_{sing}$ for the complement of $X_{reg}$ in $X$.
\medskip

{\bf Subanalytic singular simplices} are subanalytic continuous
maps $c: T_j \to X$, $T_j$ being the oriented $j$-simplex spanned by
$0,e_1,\dots,e_j$ where $e_1,\dots, e_j$ is the canonical basis of
$\R^j$.   Given a globally subanalytic set $X \subset \R^n$ we denote by $C_\bullet(X)$
the resulting chain complex (with coefficients in $\R$). We will
write $|c|$ for the support of a chain $c$ and by $\partial c$ the boundary of $c$.

\begin{dfn}\label{del pseudomanifolds}
An {\bf $l$-dimensional pseudomanifold} is a globally   subanalytic locally closed set
$X\subset \R^n$ such that
$X_{reg}$ is a manifold of dimension $l$ and $\dim X_{sing} \leq l-2$  (see Appendix I for the definition of the dimension).

A globally subanalytic subset  $Y \subset X$ is called  {\bf $(t;i)$-allowable} if
$\dim Y \cap X_{sing} <i-1$.  Define $I^{t}C_i (X)$ as the
subgroup of $C_i(X)$ consisting of those chains $\xi$ such that
$|\xi|$ is $(t, i)$-allowable and $|\partial \xi|$ is $(t, i -
1)$-allowable.

 The {\bf $j^{th}$ intersection cohomology group of maximal perversity}, denoted
$I^{t}H^j (X)$, is the $j^{th}$ cohomology group of the cochain
complex $I^tC^\bullet (X)=Hom(I^{t}C_\bullet(X);\R)$.
\end{dfn}


In this paper, we prove:

\begin{thm}\label{thm de rham}Let $X$ be a compact subanalytic pseudomanifold. For any $j$:
$$H_\infty ^j(X_{reg}) \simeq I^{t}H^j (X).$$
\end{thm}

This theorem is proved in section $4$. This requires  to investigate in details the
metric type of subanalytic singular sets. This is accomplished  in section $2$.

We then briefly recall the notion
of normalization of pseudomanifolds in section $3$.

 We will also show that the
isomorphism is given by integration on simplices (section $4.3$).
Simplices are singular and lie in $X$ (whereas forms are only
defined on $X_{reg}$) but integration is well defined and gives
rise to a cochain map if the simplices are subanalytic (see section $4.3$
for details).

\bigskip

{\bf Notations and conventions.}  We denote by $B^n(x;\ep)$ the
ball of radius $\ep$ centered at $x \in \R^n$ while $S^{n-1}(x;\ep)$ will stand
for the corresponding sphere. The symbol $|.|$ will denote the Euclidean norm while $d(.,.)$ will stand for the Euclidean distance.

We denote by
$k(0_+)$ the field of Puiseux series $\sum_{i \geq m} a_i T^{\frac{i}{p}}$, $p\in \N$, $a_i \in \R$, $i,m \in \Z$, with $\sum_{i\in \N} a_i t^i$ convergent for $t$ in a  neighborhood of zero (see Appendix II). We can order this field by setting $f\leq g$ in $k(0_+)$ if  $f(t)\leq g(t)$ for $t$ positive real number in a neighborhood of  zero. We write $T$ for the indeterminate.  The motivation for considering this field is clarified in  Appendix II.

Let $R$ stand for either $\R$ or $k(0_+)$.  By {\bf Lipschitz function}, we
will mean a function $f: A \to R$,  $A\subset R^n$, satisfying for some integer $N$:
$$|f(x)-f(x')| \leq N|x-x'|,$$
for all  $x$ and $x'$
in $A$. It is important to notice that we require
the constant to be an integer for  $k(0_+)$ is not
archimedean (see again Appendix II). A map $h:A \to R^m$ is Lipschitz if so are  all its
components; a homeomorphism $h$ is {\bf bi-Lipschitz} if
$h$ and $h^{-1}$ are Lipschitz.

Given two functions  $f , g:A \to R$, we write $f \sim g$ (and
say that $f$ is {\bf equivalent} to $g$) if there exists a positive
integer  $C$ such that $\frac{f}{C} \leq g \leq Cf$.

Given a   function $\xi: A \to \R$, we denote by $\Gamma_\xi$ its
graph and by $\xi_{|B}$ its restriction to a subset $B$ of $A$. Given two functions $\zeta$ and $\xi$ on a set $A \subset
\R^n$ with $\xi \leq \zeta$, we define the {\bf closed interval} as the set:
$$[\xi;\zeta]:=\{(x;y)\in A \times \R: \xi(x)\leq y \leq \zeta(x)\}.$$
The open and semi-open intervals are then defined analogously.

Given $A \subset \R^n$, we respectively write $cl(A)$  and $Int(A)$ for the closure and interior of $A$ (with respect to the Euclidean topology). We also define the (topological) {\bf boundary} of $A$ by  $\delta A:= cl(A) \setminus Int (A)$.

{\bf Convention.}  All the sets and mappings considered in this paper will be
assumed to be globally subanalytic (if not otherwise specified), except the
differential forms.

For the convenience of the reader,  all the necessary definitions and basic facts of subanalytic geometry may be found in two Appendixes  at the end of the paper, where references of proofs are also provided.

\section{Lipschitz retractions}
This section provides some results about the metric geometry of globally subanalytic sets.
These results   will be very
important to compute the $L^\infty$ cohomology groups later on.
Given a germ of subanalytic set $X$ at $x_0$, we shall construct a  Lipschitz
strong deformation retraction  $r_t$, $t \in [0,1]$, $r_0\equiv x_0$, $r_1=Id$, of this germ onto $x_0$ (Theorem
\ref{thm_retraction}). It is the main result of this section.

  By way of motivation for all the results of this section, let us  briefly outline  the strategy of the proof  of Theorem \ref{thm_retraction}.  Let $X$ be the germ of a singular (subanalytic) set. Replacing $X$ by $\hat{X}$ (see (\ref{eq_X_hat}) for $\hat{X}$), we may estimate the distance  to the origin  by the first coordinate.
 We proceed by induction on $n$, if $X \subset \R^n$.   The result is therefore true for $\pi_n(X)$ if $\pi_n:\R^n \to \R^{n-1}$ is the canonical projection.   By Corollary \ref{lem existence proj reg preservant xn+1} and Lemma \ref{lem_graph_proj_reg}, up to a bi-Lipschitz homeomorphism preserving the first coordinate, we know that $\hat{X}$ may be included in the graphs of finitely many Lipschitz  functions. We thus can lift the retraction obtained by induction, making use of the estimates of Lemma \ref{prop proj reg} so as to establish its Lipschitz character.

  The techniques of this
section, especially Theorem \ref{thm_retraction}, can have other
applications (see \cite{sv}).
We start by recalling some results of \cite{v1,v2}.

Given  $n>1$ and a positive
constant $M$ we set:
$$\C_n(M):=\{(x_1;x') \in \R\times \R^{n-1}:  0 \leq |x'| \leq M x_1\,\}.$$
For $n=1$, we just set $\C_1(M)=\R$.

\subsection{Regular vectors.}  In the definition below,  $R$ stands for either
$k(0_+)$ or $\R$.

\begin{dfn}\label{boule  reguliere}
Let $X$ be a   subset of $R^{n}$. An element $\lambda$ of $S^{n-1} $
is said to be {\bf regular  for $X$} if there is a  positive
real number $\alpha$ such that:
$$d(\lambda;T_x X_{reg}) \geq \alpha,$$
for any $x$ in $X_{reg}$.
\end{dfn}

Recall that the order relation in $\ko$ was defined by comparing the series on a right-hand-side neighborhood of zero. Therefore, in the above definition, the inequality means in the case $R=\ko$ that for $x \in X_{reg}$ the limit at zero of the Puiseux series $d(\lambda;T_{x} X_{reg}) $  cannot be smaller than $\alpha>0$. It is important to notice that $\alpha$ is required to be a positive {\it real number} and not a Puiseux series: it implies that the Puiseux series $d(\lambda;T_{x} X_{reg}) $ may not tend to zero at zero.

Regular vectors do not always exist, as it  is shown by the simple
example of a  circle. Nevertheless, we can  get a regular vector
without affecting the metric type of a subanalytic set:

\begin{thm}\label{thm proj reg hom pres}\cite{v1}
Let $X$ be a subset of $k(0_+)^n$ of empty interior. Then there
exists a bi-Lipschitz homeomorphism $h : k(0_+)^n \rightarrow
k(0_+)^n$ such that $e_{n}$ is regular for $h(X)$.
\end{thm}

\medskip

For instance, if $X$ is the circle (in $\ko^2$) defined by $x^2+y^2=1$ then the provided bi-Lipschitz homeomorphism may send $X$ onto a triangle (in $\ko^2$). We see (intuitively at least)  that it is not possible to require $h(X)$ to  be a smooth manifold even if so is $X$. Such a mapping $h$ is the generic fiber (see  Appendix II, section $6$) of a family of homeomorphisms sending the cylinder $(0,\ep)\times C$, where $\ep$ is a positive real number and $C$ denotes the unit circle in $\R^2$, onto the product of a triangle with the interval $(0,\ep)$. The situation  gets more difficult when $X$ is singular since it may have many different limits of tangent spaces at a singular point.

\begin{dfn}
A  map $h:\R^n \to \R^n$  is {\bf $x_1$-preserving} if it
preserves the first coordinate in the canonical basis of $\R^n$.
\end{dfn}

It is  shown in \cite{v2} that, if the considered  subset  lies in
$\C_{n}(M)$, then the homeomorphism of Theorem \ref{thm proj reg
hom pres} may be chosen $x_1$-preserving. In \cite{v2}, the result
was for semialgebraic sets.  Below, we  prove it  in the subanalytic framework.

In the proof below, we consider subsets of $\R^n$ as families of
subsets of $\R^{n-1}$ parameterized by the first coordinate. Given $t\in \R$, we write $X_{t}$ for the set
of points of $X$ having their first coordinate equal to $t$.

\begin{cor}\label{lem existence proj reg preservant xn+1}
Let $X$ be the germ at $0$ of a subset of $\mathcal{C}_{n}(M)$ of
empty interior, $M>0$. There exists a germ of
$x_1$-preserving bi-Lipschitz homeomorphism (onto its image) $h:\C_{n}(M)\to
\C_n(M)$ such that $e_{n}$ is regular for $h(X)$.
\end{cor}
\begin{proof}
Apply Theorem \ref{thm proj reg hom pres} to the generic fiber:
$$X_{0_+}:=\{x:(T;x)\in X_{k(0_+)}\},$$
where $X_{k(0_+)}$ denotes the extension of the set $X$ to $k(0_+)$ (see Appendix II).
This provides a bi-Lipschitz homeomorphism $H:\ko^{n-1}\to \ko^{n-1}$ which immediately gives rise (via the so-called transfer principle, see again Appendix II) to a $x_1$-preserving
bi-Lipschitz homeomorphism $h: (0;\ep) \times\R^{n-1} \rightarrow
(0;\ep)\times \R^{n-1}$, $(t,x) \mapsto (t,h_t(x))$, with $h_t$ bi-Lipschitz (with the same constant as $H$) for every $t<\ep$ and   such that there is a real number $\alpha>0$ such that   \begin{equation}\label{eq_e_n_regulier}
d(e_n,T_{x} h(X_{t}))\geq \alpha,
\end{equation}for any
$x \in h(X_{t})_{reg}$ and  $t$ positive small enough.  Up to a translation, we may
assume that $h_t(0)\equiv 0$ so that $h$ maps $\C_{n}(M)$ into
$\C_n(M')$, for some $M'$. Up to a $x_1$-preserving linear mapping, we may assume $M=M'$.

We now  check
 that $e_n$ is regular for the germ of $Y:=h(X)$. Suppose not. It means that  the element $(0,e_n)$ belongs to the closure of the set:
$$\{(x,u): x \in Y_{reg} \mbox{ and } u \in T_x Y_{reg} \}.$$

 As a matter of fact, by curve selection Lemma (see Appendix I),  there exists
 an analytic arc $\gamma : [0;\varepsilon]\rightarrow Y_{reg}$
 with $\gamma(0)=0$ and  $e_n \in \tau :=\lim_{t \to 0} T_{\gamma(t)}
 Y_{reg}$. On the other hand, by (\ref{eq_e_n_regulier}),  we  have $e_n \notin \lim_{t \to 0} T_{\gamma(t)}
 Y_{\gamma_{1}(t)}$. This implies that $$\tau \cap <e_{1}> ^\perp \neq \lim_{t \to 0}\, (T_{\gamma(t)}
 Y_{reg} \cap <e_{1}> ^\perp),$$  and consequently $\tau$ may not be transverse to $ <e_{1}> ^\perp
 $ (since otherwise the intersection with the limit would be the limit of the intersection), which means that  $\tau \subseteq <e_{1}> ^\perp $. This implies that the limit
 vector  $\lim_{t \to 0}
 \frac{\gamma(t)}{|\gamma(t)|} =\lim_{t \to 0}
 \frac{\gamma'(t)}{|\gamma'(t)|}\in \tau $ is orthogonal to $e_{1}$. Therefore,
  $$\lim_{t \to 0} \frac {\gamma_{1}(t)}{|\gamma(t)|}=0,$$
  in contradiction with  $\gamma(t) \in \C_n(M)$.

 Let us now show that $h$ is also Lipschitz with respect to the parameter $x_1$.    Suppose that the germ of  $h$ fails to be Lipschitz.
 In this case, the element $(0,0,0)$ belongs to the closure of the set germ:
$$\{(p,q,z): p \in \C_n(M) , q \in \C_n(M), p \neq q, z=\frac{|p-q|}{|h(p)-h(q)|} \}.$$

Then, by curve selection Lemma (see Appendix I), we can find two analytic arcs in $\C_n(M)$,
say $p(t)$ and $q(t)$, tending to zero and along which:
\begin{equation}\label{eq hyp p et q }
|p(t)-q(t)| \ll |h(p(t))-h(q(t))|.
\end{equation}

 Recall that $h$ preserves the fibers of $\pi_1$, the projection onto the first coordinate. We may assume that  $p(t)$ (and thus $h(p(t))$ too) is parameterized by its $x_1$-coordinate, i.
e., we may assume $\pi_1(p(t))=t$, $t>0$ small ($f(t):=\pi_1(p(t))$ being a real analytic function, it induces a homeomorphism in a right-hand-side neighborhood of the origin whose inverse $f^{-1}$ is a Puiseux series). As  $p(t)$ and $h(p(t))$ are Puiseux arcs in $\C_n(M)$ we
have:
\begin{equation}\label{eq proof r lips along h(p)}
|h(p(t))-h(p(t'))| \sim |t-t'|
\end{equation} and
\begin{equation}\label{eq proof r lips along p}
|p(t)-p(t')| \sim |t-t'| \leq |p(t)-q(t)|,
\end{equation}
where $t'$ denotes the first coordinate  of $q(t)$.

Therefore, by  (\ref{eq hyp p et q }) and (\ref{eq proof r lips
along h(p)}) and (\ref{eq proof r lips along p}) we have  for some constant $C \in \R$:
$$|h(p(t))-h(q(t))| \sim |h(p(t'))-h(q(t))| \sim |p(t')-q(t)|\leq C |p(t)-q(t)|,$$
%
 a contradiction. Arguing in the same way on $h^{-1}$, we could show that $h$ is
 bi-Lipschitz.
\end{proof}

 There is a close interplay between Lipschitz functions and regular vectors.

\begin{lem}\label{lem_graph_proj_reg}
Assume that  $e_n$ is regular for a set $X\subset \R^n$.  Then  $X$ is contained in  the union of the respective graphs of some Lipschitz  functions $\xi_i:\R^{n-1} \to \R$, $i=1,\dots,k$.
\end{lem}
\begin{proof}
Take a cell decomposition compatible with $X$.
 Since $e_n$ is regular for $X$, the set $X$ is  the union of some cells which are graphs (not bands, see Definition \ref{dfn_cell_decomposition}) of some analytic functions $\eta_i:D_i \to \R$, $i=1,\dots,k$, where $D_i\subset \R^{n-1}$.
  Observe that,  because $e_n$ is regular for their graph, the $\eta_i$'s have bounded derivatives.

 By Theorem $1.2$ of \cite{kp}, there is  a finite  partition  of every  $D_i$ into analytic manifolds, say $D_{i,1},\dots, D_{i,m_i}$,   and a constant $M$ such that  such that any given two points $x$ and $y$ in the same $D_{i,j}$ may be joint by an arc whose length does not exceed $M|x-y|$. This implies that any given smooth function $f:D_{i,j} \to \R$, $j\leq m_i$, which has bounded derivatives  is Lipschitz. In particular,  $\eta_i$ induces a Lipschitz function on every $D_{i,j}$, say $\eta_{i,j}$.

 Now, the lemma follows from the fact that we can extend each $\eta_{i,j}:D_{i,j} \to \R$ to a Lipschitz function $\xi_{i,j} :\R^{n-1}\to \R$ by setting:
$$\xi_{i,j}(x):=\inf \{\eta_{i,j}(y)+L_{i,j}|x-y|:y \in D_{i,j}\},$$
where $L_{i,j}$ denotes the Lipschitz constant of $\eta_{i,j}$.
\end{proof}

\subsection{Some preliminaries.}
 Before constructing the desired retraction, we need to put the set in a nice position.
 For this purpose, we will need yet another result whose proof may be found in \cite{v1} as well (Proposition \ref{lem function eq  aux distances} below).
  It is a consequence of the preparation theorem \cite{p},
   \cite{lr}, \cite{ds}.

Basically, this proposition says that  distance functions (i.e. functions of type $x \mapsto d(x,W)$, $W \subset \R^n$) may be used as a "basis of valuations", in the sense that every (globally subanalytic) nonnegative function may be
compared (up to constants) to a product of powers of distance functions (after a partition).

 We recall that, except the differential forms, all the sets and functions of this paper are assumed to be globally subanalytic.

\begin{pro}\label{lem function eq  aux distances}
Let $X\subset \R^n$  and let $\xi : X \rightarrow \R$ be a
nonnegative   function. There exists a finite partition of $X$
such that over each element of this partition the function $\xi$
is $\eqr$ to a product of powers of distances to subsets of $X$.
\end{pro}

The powers involved in the above proposition are always rational numbers.

\begin{rem}\label{rmk graphes en plus}
 We now would like to formulate two observations that will be useful  in the proof of the next lemma.
{\setlength{\leftmargini}{4pt}\begin{enumerate}\item If $X$ is the union of the graphs  of finitely many Lipschitz functions $\xi_1 , \dots ,\xi_k$ over $\R^n$ then, using the operators $\min$ and $\max$,  we may
  find an ordered family of  Lipschitz  functions  $\theta_1 \leq \dots
  \leq \theta_k$ such that $X$ is the union of the graphs of these functions.

\item Given a family of    Lipschitz functions $f_1, \dots, f_k$
defined over $\R^{n-1}$, we can find some Lipschitz functions
$\xi_1\leq \dots\leq   \xi_m$ on $\R^{n-1}$ and a cell decomposition $\D$ of
$\R^{n-1}$ such that over each $[\xi_{i|D};\xi_{i+1|D}]$,
where $D\in \D$, the family of functions $$|y-f_1(x)|,\dots ,|y-f_k(x)|,f_1(x),\dots ,f_k(x),$$ (for
$(x;y) \in [\xi_{i|D};\xi_{i+1|D}]$) is totally ordered
 (for relation
$\leq$). Indeed, it suffices to choose a cell decomposition $\D$ of $\R^{n-1}$ compatible with the sets $f_i=f_j$ and to  apply $(1)$ to   the
 functions   $f_i$,  $(f_i-f_j)$, $(f_i+f_j)$,  and  $\frac{f_i+f_j}{2}$, $i\leq k,j \leq k$.
\end{enumerate}} \end{rem}

\medskip
The lemma below somehow combines Corollary \ref{lem existence proj reg preservant xn+1} and Proposition \ref{lem function eq  aux distances} in a
single statement.  We denote by $\pi_n:\R^{n} \to \R^{n-1}$ the
orthogonal projection onto $\R^{n-1}$.

\begin{lem}\label{prop proj reg}
Given some germs  $X_1,\dots,X_s \subset \C_{n}(M)$ at $0$, there exist a germ of
$x_1$-preserving bi-Lipschitz
  homeomorphism (onto its image) $h:\C_n(M) \to \C_{n}(M)$ and a cell decomposition $\mathcal{E}$ of $\R^n$ such
  that for some representatives of the germs:
  \begin{enumerate}
  \item  $\E$ is compatible with $h(X_1),\dots, h(X_s)$ and  $h(\C_{n}(M))$. \item $e_n$ is regular for any cell of $\E$
   which is a graph (not a band, see Definition \ref{dfn_cell_decomposition}).
\item Given finitely many nonnegative   functions
$\xi_1,\dots,\xi_m$ on $\C_{n}(M)$, we may assume that on each
cell $E\subset h(\C_{n}(M))$ of $\mathcal{E}$, each function $\xi_i\circ h^{-1}$ is $\sim$
to a function of the form:\begin{equation}\label{eq
prep}|y-\theta(x)|^r a(x)\end{equation} (for $(x;y) \in \R^{n-1}
\times \R$) where
$a,\theta:\pi_n(E) \to \R$ are   functions with $\theta$ Lipschitz, $r \in \Q$. 
\end{enumerate}
\end{lem}
\begin{proof}
It will be convenient to complete the family $X_1,\dots,X_s$ by setting $X_{s+1}:= \C_{n}(M)$.
 Apply Proposition \ref{lem function eq aux distances} to the
functions $\xi_j$, $j=1,\dots,m$. This provides a partition
$E_1,\dots,E_b$   of $\C_{n}(M)$ together with some subsets of
$\C_{n}(M)$, say $W_1,\dots,W_c$, such that on each $E_i$, $i\leq
b$, each  function $\xi_j$, $j\leq m$, is equivalent to a product
of powers of functions of type  $q \mapsto d(q;W_k)$, $k \leq c$.

Possibly refining the partition $E_i$, we may assume that the $W_k$'s are unions of some elements of this partition (thanks to existence of cell decompositions, see Appendix I). Hence, on  every $E_i$, if $d(x,W_k)$ is not identically zero, then it is nowhere zero and $d(x,W_k)$ is equivalent to $d(x,\delta W_k)$.   Therefore, we may assume that the $W_{k}$'s have empty
interior, possibly replacing them with their boundaries (if a function $\xi_j$ is identically zero on
$E_i$ then $(3)$ is trivial on $E_i$).

 Apply now Corollary \ref{lem existence proj reg preservant xn+1} to the union of the  $\delta X_i$'s,
   the $\delta E_i$'s, and the $W_k$'s. This provides a germ of $x_1$-preserving bi-Lipschitz homeomorphism $h:\C_{n}(M) \to \C_n(M)$
    which maps the latter subsets
   into the union of the graphs of some   Lipschitz functions $\theta_1, \dots , \theta_d $.


 By Remark \ref{rmk graphes en plus} $(2)$ applied to the family of functions constituted by the $\theta_i$'s together with all
the $(n-1)$-variable functions $x \mapsto d(x;\pi_n( W_k\cap
\Gamma_{\theta_\nu}))$, $\nu \le d,\, k \le c$, we know that there exist a finite number of
functions $\eta_1\leq\dots\leq \eta_p$ and a cell decomposition $\D$ of $\R^{n-1}$  such that for every $D \in \D$, over each
$[\eta_{i,|D};\eta_{i+1,|D}]$, $i<p$,  the family constituted by all the $n$-variable functions
$ |y-\theta_\nu(x)| , \,  \nu \leq d,$ together with the functions $$ \, x \mapsto d(x;\pi_n( W_k\cap
\Gamma_{\theta_\nu})), \, \nu \le d,\, k \le c$$  is totally ordered (for order relation $\leq$,
considering the latter functions as  $n$-variable functions). By $(1)$ of Remark \ref{rmk graphes en plus}, we can find a totally ordered finite family $\sigma_1 \leq \dots\leq \sigma_\mu$ such that $ \cup_{i=1} ^\mu \Gamma_{\sigma_i}$ contains both the graphs of the  $\theta_i$'s and the graphs of the $\eta_i$'s.

Consider a cell decomposition $\D'$ of $\R^{n}$ compatible  with the cells of  $\D$, the sets defined by all the equations
$\sigma_{j} =\sigma_i$, $i\le  \mu$, $ j\le \mu$, as well as all the sets $h(X_j) \cap \Gamma_{\sigma_{i}}$, $j \leq s+1$, $ i \le \mu$. The graphs of the respective
restrictions of the functions $\sigma_1,\dots,\sigma_\mu$, to the sets $\pi_n(E)$, $E\in \D'$, define a cell decomposition $\E$ of $\R^n$.

For a proof of  $(1)$, take a cell $E \in \E$, $E \subset \C_n(M)$. If $E$ is a graph (not a band) then $(1)$ for $E$ follows  from the fact that $\D'$ is compatible with the    $h(X_j) \cap \Gamma_{\sigma_{i}}$'s. Assume thus that  $E$ is a band, say $(\sigma_{i|D},\sigma_{i+1|D})$ where  $i< \mu$, $D\subset \R^{n-1}$.   As $\delta h(X_j) \subset \cup_{k=1} ^\mu  \Gamma_{\sigma_k}$, for all $j$,  the set $E \cap h(X_j)$ is open and closed in $E$. Hence, if $E \cap h(X_j)$ is nonempty it is equal to $E$ ($E$ is connected). This yields $(1)$.

Observe that $e_n$ is regular for any cell of $\E$ which is a
graph, since the $\sigma_i$'s are
Lipschitz functions. This already proves that $(2)$ holds.

To prove $(3)$, fix a cell $E\subset h( \C_{n}(M))$ of $\E$ which  is  a band, say $(\sigma_{k|D},\sigma_{k+1|D})$ where  $k< \mu$, $D \subset \R^{n-1}$ ($(3)$ is trivial if $E$ is a graph). We first check that $E$ is included in  $h(E_i)$, for some $i$.    As $\delta h(E_i) \subset \cup_{k=1} ^\mu  \Gamma_{\sigma_k}$, for each $i$, the set $E \cap h(E_i)$ is open and closed in $E$. Hence, if $E \cap h(E_i)$ is nonempty it is equal to $E$.  As $h(E_1),\dots, h(E_b)$ constitute a partition of $h(\C_n(M))$, this shows that $E \subset h(E_i)$, for some $i$.

 Consequently, as $h$ is bi-Lipschitz, each $\xi_j \circ h^{-1}$ is equivalent to a product of powers of functions of type $q \mapsto d(q;h( W_i))$, $i\leq c$. It is thus enough to show (\ref{eq prep}) for these latter functions.

  As the  $\theta_\nu$'s
are Lipschitz functions, we have for any $\nu \in \{1,\dots,d\}$:
\begin{equation}\label{eq fin triang}
d(q; h(W_i) \cap \Gamma_{\theta_\nu} ) \eqr |y
-\theta_\nu(x)|+d(x;\pi_n(h( W_i) \cap \Gamma_{\theta_\nu} ))
\end{equation}
where $q=(x;y)$ in $  \R^{n-1} \times \R$.

  By construction $E \subset [\eta_{k,|A},\eta_{k+1,|A}]$, for some $k<p$ and $A \subset \R^{n-1}$.  As a matter of fact, for every $i$, the terms of the  right-hand-side are  comparable with
each other (for partial order relation $\leq$) over the cell $E$.
  Therefore, the left-hand-side is $\sim$   to one
 of them on $E$.

Note that, as each $h(W_i)$ is included in the union of the graphs of the $\theta_\nu$'s, we have:
   $$ d(q; h(W_i))
 =\underset{1 \leq \nu\leq d}{\min}
 d(q;h( W_i )\cap \Gamma_{\theta_\nu} ).$$


The latter family of functions is totally ordered over $E$. Hence, by $(\ref{eq fin triang})$, each  function $d(q;h(W_i))$
is
 equivalent over $E$ either to one of the functions $x \mapsto  d(x;\pi_n(h( W_i) \cap \Gamma_{\theta_\nu} )), $  or to
 some  function $(x;y)\mapsto  |y -\theta_\nu(x)|$, $\nu\in\{1,\dots,d\}$.
 Thus, $(3)$ holds.
\end{proof}

%

\subsection{Lipschitz retractions of subanalytic germs.}

We are now ready to construct the desired strong deformation
retraction. Given  $X \subset \R^n$  we  define:
\begin{equation}
\label{eq_X_hat}
\hat{X}:=\{(y;x)\in\R \times X: |x|=y\} .
\end{equation}
Observe that $\hat{X}$ is a subset of $\C_{n+1}(1)$.

In the theorem below we write $d_xr_t$ for the derivative of $r_t$ which exists  for $x$ generic although $r_t$ is not smooth since, like all the mappings in this paper, $r$ is implicitly assumed  to be subanalytic and thus smooth on a (subanalytic) dense subset.

 By {\bf Lipschitz deformation retraction onto $x_0$}, we mean a Lipschitz family of maps $r_t$ with $r_0(x)\equiv x_0$ and $r_1(x)\equiv x$.

\begin{thm}\label{thm_retraction}
Let $X\subset\R^n$ be locally closed and let $x_0\in  X$. Then, for any  $\ep>0$
small enough there exists a  Lipschitz deformation retraction $$r:X \cap
B^n(x_0;\ep)
\times [0;1] \to X \cap
B^n(x_0;\ep), \quad (x,t) \mapsto r_t(x),$$   onto $x_0$, preserving $X_{reg}$ for $t>0$.

Furthermore,
the derivative $d_x r_t$ tends to  $0$ as $t\to 0$ for any $x$ generic in $X_{reg}$.
\end{thm}

\begin{proof}
 We will assume
for simplicity that $x_0=0$.
 We will
actually prove by induction on $n$ the following
statements.

\bigskip

{\bf$(\textrm{A}_n)$} Let $X_1,\dots,X_s$ be finitely many subsets
of $\C_n(M)$ and let $\xi_1,\dots,\xi_m$ be  bounded
  functions on $\C_n(M)$, with $M>0$. There exists $\ep>0$ such that if we set $U_\ep:= \{x\in \C_n(M):0 \leq x_1<\ep
  \}$, there is a
Lipschitz strong deformation retraction of $U_\ep$
$$r: U_\ep\times [0;1] \to U_\ep,\quad (x,t) \mapsto r_t(x), $$ onto $0$ such that for any $j\leq s$:
\begin{enumerate}
 \item
 $r_t$ preserves   $X_j\cap U_\ep$  for $t \in (0;1]$ \item    $d_x r_t$ goes
to
 zero as $t$ tends to $0$ for any $x$ generic in $X_j\cap U_\ep$
 \item There is a constant $C$ such that
for any $i$ and any $0<t \le 1$ we have for all $x \in U_\ep$:\begin{equation}\label{eq
decroissance fn up to contant}\xi_i(r_t(x))\leq C
\xi_i(x).\end{equation}
\end{enumerate}
\bigskip

Before proving these statements, let us make it clear that this
implies the desired result. If $X\subset\R^n$ then $\hat{X}$ (see (\ref{eq_X_hat}) for $\hat{X}$) is a subset of $\C_{n+1}(1)$ to which we
can apply {\bf$(\textrm{A}_{n+1})$}. Then,  as $\hat{X}$ is
bi-Lipschitz equivalent to $X$, the result immediately ensues. Thanks to $(1)$, we may assume that the retraction preserves $X_{reg}$.

\bigskip

As the theorem obviously holds in the case where $n=1$ (with
$r_t(x)=tx$), we fix some $n>1$. We also fix some subsets
$X_1,\dots,X_s$ of $\C_n(M)$, for $M>0$,  and  some
  bounded functions $\xi_1, \dots,\xi_m:\C_n(M)\to\R$.

 \medskip

Before defining the desired map, we need some
preliminaries: we first construct a family of bounded $(n-1)$-variable
functions $\sigma_1,\dots,\sigma_p$ to which we will apply $(3)$
of {\bf$(\textrm{A}_{n-1})$}.

\medskip

Apply Lemma \ref{prop proj reg} to the family constituted by the
germs of the $X_i$'s   and the zero loci of the $\xi_i$'s.  We get a
$x_1$-preserving bi-Lipschitz map  $h:\C_n(M)\to \C_n(M)$ and a
cell decomposition $\E$ such that $(1)$ and $(2)$ of the latter
lemma hold. Moreover, thanks to $(3)$ of the latter Lemma, we
may also assume that the $\xi_i $'s
 are like in (\ref{eq prep}) on every cell.  As we may work up to a $x_1$-preserving
 bi-Lipschitz map we will identify $h$ with the identity map.

By Lemma \ref{lem_graph_proj_reg}, the union of  the cells of $\E$ for which $e_n$ is regular  may be
included in the union of the graphs of finitely many Lipschitz
functions $\eta_1 , \dots ,\eta_v$. Moreover, by Remark \ref{rmk
graphes en plus} $(1)$, we can assume $\eta_1 \leq \dots \leq\eta_v$.

In order to define the desired functions $\sigma_1,\dots,\sigma_p$, let us
fix  a cell $A$ of $\E$, and set  $A':=\pi_n(A)$, $\pi_n:\R^n \to \R^{n-1}$ denoting the projection onto the $(n-1)$ first coordinates. Choose then  $j < v$ and set $D:= (\eta_{j|A'}; \eta_{j+1|A'})$. By construction, $D$ is included in a cell of $\E$.

Since  the $\xi_k $'s
 are like in (\ref{eq prep}) on $D$, for every $k=1,\dots,m$, there exist some $(n-1)$-variable functions on $A'$, say $\theta_k$ and $a_k$, such that for $(x;y)\in D \subset \R^{n-1}\times \R$: \begin{equation}\label{eq_proof_prep}\xi_k(x;y) \sim |y-\theta_k(x)|^{\alpha_k}
a_k(x),\end{equation} where $\alpha_k$ is a rational number
(possibly negative).

As $\E$ is compatible with  the zero loci of the $\xi_k$'s, we have on  $A'$: if $\xi_k$ is not identically zero on $D$ then either
  $\theta_k \leq \eta_j$ or $\theta_k \geq \eta_{j+1}$. Fix $k$ with $\xi_k\neq 0$ on $D$. We will  assume for simplicity  that  $\theta_k \leq \eta_j$.

It means that on $D$:
\begin{equation}\label{eq min}\xi_k(x;y) \sim \min ((y-\eta_j(x)) ^{\alpha_k} a_k(x)
;(\eta_j(x)-\theta_k(x)) ^{\alpha_k} a_k(x)),\end{equation}  if
$\alpha_k$ is negative, and
\begin{equation}\label{eq max}\xi_k(x;y) \sim \max ((y-\eta_j(x)) ^{\alpha_k} a_k(x)
;(\eta_j(x)-\theta_k(x)) ^{\alpha_k}  a_k(x)),\end{equation} in the
case where $\alpha_k$ is nonnegative.

We are now ready to define the desired family
$\sigma_1,\dots,\sigma_p$ of $(n-1)$-variable
functions. 
 We first set for $\xi_k \neq 0$ on $D$:
\begin{equation}\label{eqdefkappak}\kappa_k(x):=|\eta_{j}(x)-\theta_k(x)|^{\alpha_k} a_k(x).\end{equation} 
Since $\xi_k$ is bounded, by (\ref{eq prep}), this defines a bounded function. Complete the family $\kappa$ by adding the functions $\min(f;1)$
where $f$ describes all the $(\eta_{j+1}-\eta_j)a_k$'s.

 Doing this
for all the cells  $A \in \E$ and integers $j<v$, and collecting all the respective families $\kappa$ obtained in this way, we eventually get a family of bounded
functions $\sigma_1,\dots,\sigma_p$.

We now turn to the construction of the desired retraction. Consider a
cylindrical cell decomposition $\mathcal{D}$ compatible with the cells of
$\E$ and the graphs of
the $\eta_j$'s. 
Apply the induction hypothesis to the  family of sets $\pi_n (D) \cap \C_{n-1}(M)$, $D \in \mathcal{D}$. This provides a   deformation retraction $r:V_\ep \times [0;1]\to V_\ep$, $\ep >0$, where $V_{\ep}:=\{x \in \C_{n-1}(M):0\le x_1<\ep \}$.

We are going to lift $r$ to a retraction of $[\eta_{1|V_\ep},\eta_{v|V_\ep} ]$. Thanks to the
induction hypothesis, we may assume that the functions
$\sigma_1,\dots,\sigma_p$, as well as the
$(\eta_{j+1}-\eta_j)$'s and the functions $x \mapsto
\xi_i(x;\eta_j(x))$  satisfy (\ref{eq decroissance fn up to
contant}). 

Now, we may lift $r$ as follows. On $(\eta_j;\eta_{j+1})$, $j=1,\dots,v-1$, we set
$$\nu(q):= \frac{y-\eta_{j}(x)}{\eta_{j+1}(x)-\eta_{j}(x)},$$ if
$q=(x;y) \in (\eta_j;\eta_{j+1})\subset \R^{n-1}\times \R$, and then
$$\widetilde{r}_t(q):=(r_t(x);\nu(q)(\eta_{j+1}(r_t(x))-\eta_j(r_t(x))) +\eta_j(r_t(x))).$$
 This mapping is then easily extended continuously on each $\Gamma_{\eta_j}$ by setting if $q=(x;\eta_j(x))$:
 $$\widetilde{r}_t(q):=(r_t(x);\eta_j(r_t(x))).$$

For any $j$, the mapping $\tilde{r}$ maps linearly the segment
$[\eta_j(x);\eta_{j+1}(x)]$ onto the segment
$[\eta_j(r_t(x));\eta_{j+1}(r_t(x))]$. Thanks to the induction
hypothesis, the inequality (\ref{eq decroissance fn up to
contant}) is fulfilled by the function $(\eta_{j+1}-\eta_j)$.
Therefore, as $r$ is Lipschitz, we see that $\widetilde{r}$ is
Lipschitz as well. As $\widetilde{r}$ preserves the cells of $\E$, it preserves the $X_j$'s, the zero loci of the $\xi_k$'s, and $U_\ep$.

 We
have to check that the $\xi_k$'s fulfill (\ref{eq decroissance fn
up to contant}) along the trajectories of $\widetilde{r}$. We
check it on a given cell $E$ of $\D$. If $E\subset \Gamma_{\eta_j}$ for some $j$, this follows from the induction
hypothesis since we have assumed that the functions $x \mapsto
\xi_k(x;\eta_j(x))$ satisfy  (\ref{eq decroissance fn up to
contant}) on $E$.

 Otherwise, there exists $j$ such that $E$ sits in $(\eta_j;\eta_{j+1})$. Fix an integer $1 \leq k
\leq m$. On the cell $E$, the function $\xi_k$ may be estimated as
in (\ref{eq_proof_prep}). By the induction hypothesis we know that
$\kappa_k$ (see (\ref{eqdefkappak}),  if $\xi_k = 0$ on $E$ then  (\ref{eq decroissance fn
up to contant}) is trivial for $\xi_k$) satisfies (\ref{eq
decroissance fn up to contant}). 

If a function $\xi$ is bounded and if
$\min(\xi;1)$ satisfies $(\ref{eq decroissance fn up to contant})$
then $\xi$ satisfies this inequality as well.
We will therefore check (\ref{eq decroissance fn up to contant}) for $\min(\xi_k;1)$.

  Observe also that if two given functions $\xi$ and
$\zeta$ both satisfy (\ref{eq decroissance fn up to contant}) then
$\min(\xi;\zeta)$ and $\max(\xi;\zeta)$ both satisfy this
inequality as well.   Hence, by (\ref{eq min}) and (\ref{eq max}),
  it is enough to show that the functions $\min((y-\eta_j(x))^{\alpha_k} a_k(x);1)$
   and the functions $|\theta_k-\eta_j| ^{\alpha_k} a_k$ satisfy (\ref{eq decroissance fn
up to contant}).
 The latter functions  are nothing but  the $\kappa_k$'s for which we already have seen that this inequality is true.
 Let us focus on the former functions.

For simplicity we set $$F(x;y):=(y-\eta_j(x))^{\alpha_k} a_k(x)$$
and $$G(x):=(\eta_{j+1}-\eta_j)(x)^{\alpha_k}  a_k(x).$$

We have to show the
desired inequality for $\min(F;1)$. We have:
\begin{equation}\label{eq F G}F(x;y)=\nu(q)^{\alpha_k} \cdot G(x).\end{equation}

 Remark that the function
$\nu(\widetilde{r}_t(q))$ is constant
with respect to $t$. This
implies that:
\begin{equation}\label{eq2 F G}F(\widetilde{r}_t(q))=\nu(q)^{\alpha_k} \cdot G(r_t(x)).\end{equation}

We assume first that $\alpha_k$ is negative. Thanks to the
induction hypothesis ($\min(G;1)$ is one of the $\sigma_i$'s) we
know that for some constant $C$ we have for all $x$ in $\pi_n(E)$:
$$\min(G(r_t(x));1)\leq C \min (G(x);1).$$

This implies (multiplying by $\nu^{\alpha_k}$ and applying
(\ref{eq F G}) and (\ref{eq2 F G})) that for $q \in E$:
$$\min(F(\widetilde{r}_t(q));\nu^{\alpha_k}(q);1)\leq C\min (F(q);\nu^{\alpha_k}(q);1). $$
But, as $\alpha_k$ is negative,
$\min(F;\nu^{\alpha_k};1)=\min(F;1)$, which yields the desired
inequality for $\min(F;1)$, as required.

\medskip

We now assume that $\alpha_k$ is nonnegative. 
Thanks to (\ref{eq F G}) and (\ref{eq2 F G}), it
actually suffices to show the desired inequality for $G$. But, as $\xi_k$ is bounded,
by (\ref{eq
max}) so is $G$,  and the result follows from the induction hypothesis  since $\min(G;1)$ is one of the $\sigma_i$'s (as $G$ is bounded and $\min(G,1)$ satisfies (\ref{eq decroissance fn up to contant}) then $G$ satisfies this inequality as well). This yields (\ref{eq decroissance fn up to
contant}) along the trajectories of $\widetilde{r}$.

We now check that $d_q \widetilde{r}_t$ tends to zero when $t$ goes to zero. It
follows from the induction  hypothesis that $d_x r_t$ goes to zero
as $t$ goes to zero, for any $x$ generic. As the $\eta_i$'s have bounded derivatives,
this already proves for almost every $x$:
\begin{equation}\label{eq der de eta} \underset{t \to 0}{\lim}
\;\: d_x [(\eta_i -\eta_{i+1}) \circ r_t] =0.\end{equation}

On the other hand, a straightforward computation shows that for
$q=(x;y)$:
$$|d_{\widetilde{r}_t(q)} \nu| \leq  \frac{C}{|\eta_i(x)-\eta_{i+1}(x)|} ,$$ which,
together with  (\ref{eq der de eta}) and (\ref{eq decroissance fn
up to contant}) for $(\eta_{i+1}-\eta_i)$,  implies that $d_q
\widetilde{r}_t$ tends to zero as $t$ goes to zero.
\end{proof}

\begin{rem}\label{rem lt}
The mapping $r_t$ could be proved to be bi-Lipschitz for every $t>0$. The Lipschitz constant of $r_t^{-1}$ may of course tend to infinity as $t$ goes to zero.   We nevertheless could have a control on the way distances are contracted by $r_t$, similarly as in \cite{v1,v2}.   We could show that for a suitable basis of unit $1$-forms $\theta_1,\dots,\theta_n$ and some functions $\varphi_1,\dots,\varphi_n$ on $\R^n$ such that almost everywhere on $X_{reg} \cap B^n(x_0;\ep) \times [0,1]$:  $$r_t^*\rho (x) \approx \sum_{i=1} ^n  \varphi_i (x;t)^2 \theta_i^2(x),$$
 where $\rho$ is the metric of $\R^n$ and $r^*_t\rho$ its pull-back. Similarly as in \cite{v1}, the functions $\varphi_i$, which are the contractions of the metric that $r$ operates,  could be  expressed as powers,  products, and sums of distance functions in $X$ (i. e.  $x \mapsto d(x;W)$ with $W \subset X\cap B^n(0;\ep)$) and  the function $(x;t)\mapsto t$.
These powers may be negative which makes it difficult to get decreasing functions and accounts for the difficulty we have in the proof of the above theorem.
\end{rem}

\section{Normal pseudomanifolds.} In this section, we shall also deal with {\it topological} pseudomanifolds. Given $X\subset \R^n$, denote by $X^0_{reg}$ the set of points of $X$ near which $X$ is a $C^0$-manifold (of any dimension).  We say that a locally closed set  $X$ is an $l$-dimensional {\bf topological pseudomanifold}  if $\dim X\setminus \xb^0<l-1$ and if  $\xb^0$ is an $l$-dimensional manifold.
\begin{dfn}\label{dfn normal}
An $l$-dimensional topological pseudomanifold $X$ is called {\bf normal} if
for any $x$ in $X$, $\dim H_l(X;X\setminus \{x\})=1$.
 \end{dfn}

We shall recall some basic facts about normal pseudomanifolds. These  may
be found in  \cite{gm}  (section $4$) and make normalizations very
useful to investigate intersection homology in the maximal
perversity. Observe that if $X\subset \R^n$ is a normal topological pseudomanifold which is connected then $H_l (X)=\R$, since if
there were two generators, say $\sigma$ and $\tau$, $\dim |\sigma|
\cap |\tau|<l$, we would have $H_l (X;X \setminus x) \neq \R$ at any
point of the intersection of the supports.

\medskip

  The main interest of normal spaces lies in the following Lemma. Denote by $L(x;\xb)$ the set $S^{n-1}(x;\varepsilon) \cap \xb$. It is well known that the topology of   $L(x;\xb)$ is independent of $\varepsilon>0$  small enough.

\begin{lem} \cite{gm}
A  topological pseudomanifold $X\subset \R^n$ is normal if and only if $L(x;\xb^0)$ is
connected at any point of $X \setminus X _{reg} ^0$.
\end{lem}

See for instance \cite{gm} section $4$ for a proof.  The very significant advantage of normal pseudomanifolds lies in the following proposition.

\begin{pro}\label{pro ih et x normal} \cite{gm}
Let $X$ be a normal topological  pseudomanifold. The mapping $\alpha: I
^{t}H_j(X) \to H_j(X)$, induced by the inclusion between the chain
complexes, is an isomorphism for all $j$.
\end{pro}

 \medskip

\subsection{Normalizations of pseudomanifolds.} We shall need some basic facts about normalizations. 

\begin{dfn}\label{dfn normalization} A {\bf normalization} of the topological pseudomanifold $X$ is a normal topological
pseudomanifold $\widetilde{X}$ together with a finite-to-one continuous mapping
 $\pi: \widetilde{X}\to X$ such that, for any $p$ in
$X$, $$\pi_*: \underset{q \in \pi^{-1}(p)}{\oplus}
H_l(\widetilde{X};\widetilde{X} \setminus q) \to H_l (X;X\setminus
p)$$ is an isomorphism.
\end{dfn}

We are going to see that normalizations are useful to compute intersection homology in the maximal perversity.

\begin{pro}\label{pro_Linfty_normalization}
Every  pseudomanifold  $X$ admits a normalization $\pi:\widetilde{X} \to X$. The mapping $\pi$ then induces a homeomorphism above the regular locus of $X$.
\end{pro}
\begin{proof}
We follow the construction of \cite{gm}. Consider a triangulation of $X$ (since $X$ is globally subanalytic, it admits a $C^0$ triangulation, see Appendix I), i.e., a homeomorphism  $T :K \to X$, with $K$ finite union of open simplices.

Let $L$ be the disjoint union of all  the closures in $K$ of the $l$-dimensional  open simplices of $K$ (where $l$ is the dimension of $X$).  Identify  the closure in $K$ of  two  $(l-1)$ open faces   of two elements
    of $L$ if these two faces coincide in $K$.  This provides a simplicial complex $\tilde{X}$. Denote then by $\pi:\tilde{X}\to X$ the  map induced by $T$.

Observe that by construction the mapping $\pi$ is a homeomorphism on the complement in $X$ of the $(l-2)$-skeleton. It is thus easily checked  from the definition that the mapping $\pi$ induces a homeomorphism above  $X_{reg}$ and that  $L(x,\tilde{X}^0_{reg})$  is connected at singular points.
\end{proof}

\begin{rem}
It is possible to see that  the normalization of a pseudomanifold is unique, up to a homeomorphism.
\end{rem}

\medskip

It is not difficult to see from their construction that
normalizations must identify $(t;i)$-allowable chains of
$\widetilde{X}$ with   $(t;i)$-allowable chains of $X$, which
implies  that they yield  an isomorphism between the intersection
homology groups (see \cite{gm}):

\begin{pro}\label{pro ih et normalization} \cite{gm}
Let  $\pi:\tilde{X} \to X$ be a  normalization of $X$. Then, for
any $j$ the induced map $\pi_*:I^{t}H_j (\tilde{X}) \to I^{t}H_j
(X)$ is an isomorphism.
\end{pro}

\bigskip

\section{Computation of the $L^\infty$ cohomology groups}
This section  proves the main result of this
paper, Theorem \ref{thm de rham}.

\subsection{Weakly differentiable forms.} For technical reasons,
we will need to work with non smooth forms, which are  weakly
differentiable, i.e., differentiable as distributions. Therefore,
the first step is to prove that the bounded  weakly differentiable
forms give rise to the same cohomology theory. We will follow an argument similar to the one used by Youssin in \cite{y}.

Let $M$ be a smooth manifold.  We denote by $\Lambda_{0} ^{j}(M)$ the set of
$C^2$ $j$-forms on $M$ with
compact support.

\begin{dfn}
Let $U$ be an open subset of $\R^n$. A continuous differential
$j$-form $\alpha$ on $U$ is called {\bf weakly differentiable} if
there exists a continuous $(j+1)$-form $\omega$ such that for any form
$\varphi\in \Lambda_{0} ^{l-j-1} (U)$:$$\int_{U} \alpha \wedge
d\varphi =(-1)^{j+1}\int_{U} \omega \wedge \varphi.$$ The form
$\omega$ is then called {\bf the weak exterior differential of
$\alpha$} and we write $\omega=\overline{d} \alpha$.   A
continuous differential $j$-form $\alpha$ on $M$ is called {\bf
weakly differentiable} if  it  gives rise to   weakly
differentiable forms via the coordinate systems  of $M$.
\end{dfn}

We denote by $\overline{\Omega}_\infty ^{j}(M)$ the set of weakly
differentiable forms which are bounded and which have a  bounded
weak exterior differential. They constitute a cochain complex
whose coboundary operator is $\overline{d}$. We denote by
$\overline{H}_\infty ^\bullet (M)$ the resulting cohomology
groups.

It is well known that if $\omega$ is smooth then it is weakly differentiable and $d\omega=\overline{d}\omega$. Therefore $\Omega_\infty ^{j}(M)\subset \overline{\Omega}_\infty ^{j}(M)$. Moreover, every $L^\infty$ weakly differentiable form may be approximated (for the $L^\infty$ norm) by smooth bounded forms (with approximation of the differential if it is $L^\infty$). Consequently, any weakly differentiable $0$-form $\omega$ satisfying $\overline{d} \omega =0$ is constant.

We shall see that smooth and weakly differentiable forms give rise to isomorphic cohomology theories. The lemma below addresses the case of compact manifolds with boundary.

  Given a smooth manifold  with boundary $K$,  we write  $H_{dR} ^j (K)$ for the de Rham cohomology of $K$, i.e., the cohomology of the $C^\infty$ differential forms on $K$.

\begin{lem}\label{lem_K_avec_bord}Let $K$ be a compact manifold with boundary. The mapping $H_{dR} ^j (K) \to \overline{H}^j _\infty (K \setminus \partial  K)$ induced by  the inclusion between the respective cochain complexes is an isomorphism.   \end{lem}
\begin{proof}
As the smooth forms on $K$ satisfy Poincar\'e Lemma (see for instance \cite{bl}), they give rise to a fine torsionless resolution of the constant sheaf.  By the
uniqueness of the map between sheaf cohomology theories with
coefficient in sheaves of $\R$-modules, it is enough to show
Poincar\'e Lemma for weakly differentiable forms, i.e., it is enough to show that every point of $K$ has a contractible neighborhood $U$ in $K$ such that for any $\omega \in  \overline{\Omega} ^j _{\infty}(U\setminus \partial K) $, $j>0$,  there is $\alpha \in\overline{\Omega} ^{j-1} _{\infty}(U\setminus \partial K)$, such that $\overline{d} \alpha =\omega$.

 Poincar\'e  Lemma for $\overline{\Omega} ^j _{\infty}(K\setminus \partial K)$ may be either derived  by following the same argument as for the smooth forms on compact manifolds with boundary or directly deduced from the proof of Theorem \ref{thm x normal Poincare lemma} which actually applies to any weakly differentiable bounded $j$-form $\omega$  (this theorem indeed states a  more difficult result since it deals with every subanalytic set, possibly singular).
\end{proof}

 For noncompact manifolds, we can now prove the following:

\begin{pro}\label{pro l1 isom smooth}
For any $C^\infty$ manifold $M$ (without boundary), the inclusion $\Omega_{\infty}
^\bullet (M) \hookrightarrow \overline{\Omega}_{\infty} ^\bullet (M)$ induces
isomorphisms on the cohomology groups.
\end{pro}
\begin{proof}
It is enough to show that, for any form $\alpha \in
\overline{\Omega}_{\infty} ^{j} (M)$  with $\overline{d}\alpha \in
\Omega_{\infty} ^{j+1} (M)$ (i. e. $\alpha$ is weakly differentiable and $\overline{d}\alpha$ is  {\it smooth}),
  there exists  $\theta \in \overline{\Omega}_{\infty} ^{j-1} (M)$ such that $(\alpha+\overline{d}\theta)$ is
  $C^\infty$ (if $j=0$ then $\theta\equiv 0$).

Choose a sequence of compact smooth manifolds
with boundary $K_i \subset M$, $i\in \N$, such that for each
 $i\geq 0$, $K_i$ is included in the interior of
$K_{i+1}$ and $\cup K_i=M$.

 Fix a form $\alpha \in \overline{\Omega}_{\infty}
^{j} (M)$  with $\overline{d}\alpha \in \Omega_{\infty} ^{j+1}
(M)$.   We are going to construct a sequence $(\theta_i)_{i \in \N}$ in $\overline{\Omega}_{\infty} ^{j-1}(M)$
  such that for every  $i \in \N$, we have
  $supp \; \theta_i \subset Int(K_{i}) \setminus K_{i-2}$ as well as $|\theta_i|_\infty+|\overline{d}\theta_i|_\infty\leq
  1$ and such that the form $\alpha_i:=\alpha+ \sum_{k=0} ^i\overline{d}\theta_k$ is smooth in a neighborhood of $K_{i-1}$.

  Before defining inductively the $\theta_i$'s, observe that $\theta:=\sum_{i=0}
  ^\infty  \theta_i$ is the desired form (this sum is locally
  finite).

We now define the
$\theta_i$'s by induction on $i$.
Let us assume that $\theta_0,\dots,\theta_{i-1}$ have been constructed, $i \geq 1$
(we may set $K_{-1}:=K_{-2}:=\emptyset$).  We will also argue by induction on $j$. For $j=-1$, both cochain complexes vanish and the result is clear.

 Observe that by Lemma \ref{lem_K_avec_bord},
there exists a smooth $j$-form $\beta$ on
$K_i$ such that $d\beta=\overline{d}\alpha_{i-1}$. It means
that $(\alpha_{i-1}-\beta)$
   is $\overline{d}$-closed, and thus again by Lemma \ref{lem_K_avec_bord} there is a smooth $j$-form $\beta'$ on  $K_i$ such that
   \begin{equation}\label{eq_def_gamma}
\alpha_{i-1}-\beta=\beta' +\overline{d} \gamma,
\end{equation} with $\gamma \in
   \overline{\Omega}^{(j-1)}_{\infty} (K_i)$.

 Thanks to the induction on $i$, we know that
     there exists an open  neighborhood $V$ of $K_{i-2}$ on which
     $\alpha_{i-1}$ is smooth. This implies that $\overline{d} \gamma$ is smooth on $V$.
      Therefore, applying the induction hypothesis to $\gamma$ (which is a $(j-1)$-form),  we can add a weakly exact
   form $\overline{d}\sigma$ to $\gamma$ to get a form smooth on a neighborhood of $K_{i-2}$. Multiplying $\sigma$ by a smooth function which has  compact support included in $V$ and  which is $1$  on a neighborhood $W\subset V$ of $K_{i-2}$, we get a form $\sigma'$ on $M$ such that $(\overline{d}\sigma'+\gamma)$ is smooth on $W$.  It means that
   we can assume that $\gamma$ is smooth on an open neighborhood $W$ of $K_{i-2}$. We will assume
   this fact without changing notations.

    By means of convolution products with a bump function, for any
    $\ep>0$,
     we may construct a smooth form  $\gamma_\varepsilon$ such that $|\gamma_\varepsilon-\gamma|_\infty\leq \varepsilon$ and
     $|d\gamma_\varepsilon-\overline{d}\gamma|_\infty\leq
     \varepsilon$ on $K_i$.

     Consider a smooth function $\phi$ which is $1$ on a neighborhood of
     $(M \setminus W)\cap K_{i-1}$ and with support in $(K_i \setminus \partial K_i) \setminus
     K_{i-2}$. Then, set:
\begin{equation}\label{eq_def_theta_i}
\theta_i(x):=\phi(x)(\gamma_\ep-\gamma)(x).
\end{equation}

 If
$\varepsilon$ is chosen small enough $|\theta_i|_\infty+|\overline{d}\theta_i|_\infty \leq 1$. On  a
neighborhood of
     $(M \setminus W)\cap K_{i-1}$, because $\phi\equiv 1$,
we have by (\ref{eq_def_gamma}) and (\ref{eq_def_theta_i}):
$\alpha_{i-1}+\overline{d}\theta_i=\beta+\beta'+d\gamma_\varepsilon$
which is clearly smooth. The form
$(\alpha_{i-1}+\overline{d}\theta_i)$ is smooth on $W$ as well,
since
     $\alpha_{i-1}$ and $\theta_i$ are both smooth.
      \end{proof}

\subsection{Proof of the De Rham Theorem for $L^\infty$ cohomology.}
We are now ready to prove Poincar\'e Lemma for $L^\infty$ cohomology.

\begin{thm}\label{thm x normal Poincare lemma}(Poincar\'e Lemma for $L^\infty$ cohomology)
Let $X \subset \R^{n}$ be locally closed and let $x_0 \in X$. There exists $\ep>0$ such that  for  any closed form $\omega\in
\Omega_\infty ^j(B^n(x_0,\ep) \cap \xb)$, $j\geq 1$, we can find
$\alpha \in \Omega_\infty ^{j-1}(B^n(x_0,\ep) \cap X_{reg})$,  such that
$\omega=d\alpha$.
\end{thm}
\begin{proof}
Let $r:X\cap B^n(x_0,\ep)\times [0;1] \to X \cap B^n(x_0,\ep)$ be the map obtained by applying
 Theorem \ref{thm_retraction} to $X$.
 For simplicity, as our problem is local, we will identify  $X$ with  the subset $X \cap B^n(x_0;\ep)$.
 Let $\omega \in
\Omega_\infty ^j(X_{reg})$, with $j\geq 1$ and $d \omega =0$. By Proposition \ref{pro l1 isom smooth}, it is enough to find $\alpha \in \overline{\Omega}_\infty ^{j-1}(X_{reg})$ satisfying $\overline{d} \alpha =\omega$.

The problem is that $ r$ may fail to be weakly smooth.
To overcome  this difficulty, we shall work with an approximation
of $r$. We need to be particular since we wish to preserve the
property that the derivative of $r_t$ goes to zero (pointwise and
generically) as $t$ goes to zero.

 Consider a sequence of compact subsets, $(K_i)_{i\in \N}$,  such that  $\cup K_i=X_{reg}\times (0;1)$ and  $K_i \subset Int(K_{i+1})$, for any $i$. Let $Y$ be the set of points of $X_{reg}\times (0;1)$ at which $r$ fails to be smooth. Define then  a sequence of compact subsets for $i \ge 1$: $$L_i:={\{q \in K_i:d(q;Y)}\geq 1/i\}.$$   
Define also $$X':=cl(Y)\cap (X_{reg}\times \{0\})$$  and observe that, since $Y$ is of positive codimension in $X_{reg}\times (0,1)$,  $X'$ is of positive codimension in $X_{reg}$ (we will consider it as a subset of $X_{reg}$).

 As $r$ is continuous, we may choose, for a given  $\ep_i>0$, a $C^\infty$ approximation $r_i$ (not necessarily subanalytic) of $r$ on $K_i$ satisfying for any $x$ in this set:
$$|r_t(x)-r_{i,t}(x)| \leq \ep_i.$$  Furthermore, as $r$ is smooth on $L_i $,  we may require that on this set  \begin{equation}\label{eq_d_xr_et_d_xrprime}
|d_xr_{i,t}-d_xr_t|\leq \ep_i.
\end{equation}
Moreover, as the first derivative of $r$ is bounded, the first  derivative of  $r_i$ may  be assumed to be bounded as well.

 Let $(\varphi_i)_{i \in \N}$ be a partition of unity subordinated to the covering $(Int(K_{i+2})\setminus K_{i})_{i\in \N} $ of $X_{reg}\times (0;1)$. Set $r':=\sum \varphi_i r_i$.  If the sequence $\ep_i$ is decreasing fast enough, then a straightforward computation shows that the first derivative of $ r'$ is  bounded above.

Furthermore, given any positive continuous function $\ep
:X_{reg}\times(0;1) \to \R$, we can have if the sequence $\ep_i$
is decreasing fast enough:
\begin{equation}\label{eq_approx}|r'_t(x)-r_t(x)|\leq \ep(x;t).\end{equation}
Finally, we shall check that $d_x r'_{t}$ tends to zero as $t$ goes to zero for $x \notin X'$. Fix $x$ in $X_{reg}\setminus X'$.  There exists $a>0$ such that $\{x\}\times  [0;a]$ does not meet $Y$. It means that for any $i$ large enough \begin{equation}\label{eq_K_et_L}
K_i\cap (\{x\}\times  [0;a])=L_i\cap (\{x\}\times  [0;a]).
\end{equation}
For $t$ small enough, if $\varphi_i(x;t)\neq 0$ then $(x;t) \in K_{i+2}$ and thus  by (\ref{eq_K_et_L}) belongs to $L_{i+2}$. By  (\ref{eq_d_xr_et_d_xrprime}), this implies that if  $ \ep_i$ tends to zero fast enough, $\lim_{t \to 0} |d_x r'_t|=0$ for  every $x \in X_{reg}\setminus X'$. We also see  for the same reason  that $d_x r'_t$ tends to the identity as $t$ goes to $1$ (for almost every $x$).

Let $\pi:W \to X_{reg}$  be a retraction where $W$ is a tubular
neighborhood of $X_{reg}$. Taking $W$  small enough, we may assume
that $\pi$ has bounded first partial derivatives. 
By (\ref{eq_approx}), $r'_t(x)$ belongs to $W$ if the function  $\ep$ is decreasing fast enough. Hence, composing $r'$ with $\pi$ if necessary we may assume that $r'$ preserves $X_{reg}$.  We will assume this without changing notations.

Define  two $L^\infty$ forms $\omega_1$ and $\omega_2$ on $X_{reg} \times (0;1]$ by:
$$r'^*\omega:=\omega_1+dt\wedge \omega_2.$$
Now, we may set:
$$\alpha(x):=\int_0 ^1 \omega_2 (x;t) dt.$$
As $\omega$ is $L^\infty$ and $r'$ has bounded first partial derivatives, the
form $\alpha$ is clearly bounded. By Lebesgue's dominated
convergence theorem, it is continuous. We claim that
it is weakly
differentiable and that $\overline{d}\alpha=\omega$.

Let us fix a  $C^2$-form $\varphi \in \Lambda^{m-j} _0 (\xb)$ with
compact support. We have, by definition of $\alpha$:
\begin{equation}\label{eq_der_faible_calcul}\int_{\xb}
\alpha \wedge d\varphi=
  \int_{\xb}   \int_0 ^1 \omega_{2}\wedge d\varphi = \lim_{t \to 0} \int_{\xb \times [t;1]}
r'^*\omega  \wedge d\varphi  .\end{equation}

As $r'^*\omega$ is closed, by Stokes' formula we
have:
$$ \int_{\xb \times [t;1]}
 r'^*\omega  \wedge d\varphi =(-1)^j \int_{x\in\xb }\omega(x) \wedge
\varphi(x) - (-1)^j\int_{x\in \xb }\omega_1(x;t) \wedge \varphi(x),$$
since $\lim_{t\to 1}r'^*\omega(x;t)=\omega(x)$ for any $x \in \xb$. Recall
that $d_xr'_t$ tends to
 zero as $t$ goes to $0$ for almost every $x$. This implies that $\omega_1(x;t)$ goes to zero as $t$ goes to
 $0$. Hence, passing to the limit we get:
 $$\int_{\xb} \alpha \wedge
d\varphi=(-1)^{j}\int_{\xb} \omega \wedge \varphi,$$ as required.
\end{proof}

\medskip

\begin{proof}[Proof of Theorem \ref{thm de rham}]
Let $\pi :\tilde{X}\to X$ be a normalization of $X$ (see Proposition \ref{pro_Linfty_normalization}).
 Let us define a presheaf  on $\tilde{X}$ by
$$\tilde{\Omega}_\infty ^j(U):=\Omega_\infty ^j(\pi(U)\cap
X_{reg}),$$ for every open set $U$ of $\tilde{X}$ ($\pi$ is a homeomorphism above $X_{reg}$). For every $j$, this
presheaf   immediately gives rise to a sheaf that we will denote
by $\mathcal{F}_\infty^j$. We will write  $\mathcal{F}_{\infty,x_0} ^j $  for the stalk of $\mathcal{F}_\infty ^j$ at $x_0 \in \tilde{X}$, i.e., the vector space obtained after identifying two sections which coincide near $x_0$.   As $\tilde{X}$ is compact, any global section
of $\mathcal{F}_\infty^j$ is bounded, so
that, since $\pi$ induces a homeomorphism above $X_{reg}$:\begin{equation}\label{eq_sect_globales}
\mathcal{F}^\bullet _\infty (\tilde{X})\simeq
\Omega_\infty ^\bullet (X_{reg}),
\end{equation}
 as cochain
complexes.

We denote by $\R_X$ the constant sheaf on $X$.
Let $x_0 \in \tilde{X}$ and set $U^\ep:=B^n(x_0,\ep)\cap \tilde{X}$.
 As $\pi$  is a  normalization,
$\pi(U^\ep)\cap X_{reg}$ is connected,
which means that $H^0 _\infty(X_{reg}\cap
\pi(U^\ep))=\R$, for $\ep>0$ small enough.

Moreover, by Theorem \ref{thm x normal Poincare lemma}, for $j>0$,  the   germ at $\pi(x_0)$ of a smooth bounded closed $j$-form $\omega$ on $\pi(U^\ep) \cap \xb$ is the exterior differential of  the germ  of a form $\alpha \in \mathcal{F}_{\infty,x_0} ^{j-1}$.
Therefore, the
sequence:
\begin{equation}\label{eq fine res covan}  0\longrightarrow \R_X \overset{d}{\longrightarrow} \mathcal{F}^0 _\infty  \overset{d}{\longrightarrow} \mathcal{F}^1_\infty \overset{d}{\longrightarrow}\dots \end{equation}
is a fine torsionless resolution of the constant sheaf. By
classical arguments of sheaf theory (see for instance \cite{w}),
the latter exact sequence of sheaves implies via (\ref{eq_sect_globales}) that  $H_\infty ^j(X_{reg})$ is isomorphic to the  singular
cohomology of $\tilde{X}$. But then, by Propositions \ref{pro ih et x
normal} and \ref{pro ih et normalization}, we get:\begin{equation}\label{eq proof isom de
rham}H_\infty ^j(\xb)\simeq H^j(\tilde{X}) \simeq I^{t}H
^j(\tilde{X})\simeq I^t H^j(X).  \end{equation} \end{proof}

\subsection{Integration on subanalytic singular simplices.}  We are going to prove that the isomorphism is provided by integrating the forms on the allowable
 chains.  We first check that integration gives rise to a well defined cochain map.
  This may be done because we restrict ourselves to the  $t$-allowable {\it subanalytic} singular cochains.
Let $X$ be a compact  pseudomanifold.


Let $L\subset X_{reg}$ be an oriented    manifold of dimension $j$
with $cl(L)$ $(t;j)$-allowable i. e.: $$dim \,cl(L) \cap X_{sing}\leq(j-2).$$

Set $\partial  L := cl(L)\setminus L$.  Then, for any given
$\omega$ in $\Omega_\infty^{j-1} ( X_{reg})$, $\int_L  d\omega$
and $\int_{(\partial L)_{reg}} \omega$ are well defined since
$\omega$ is continuous almost everywhere on $(\partial L)_{reg}$ and
bounded. We start by recalling a version of Stokes' formula proved
by \L ojasiewicz in \cite{l} who generalized a formula of Paw\l ucki \cite{pawluckistokes}.

\begin{lem}\label{lem lojasiewicz} \cite{l} Take $L$ and $\omega$ as in the above paragraph. Then: \begin{equation}\label{eq stokes l}
\int_{(\partial L)_{reg}} \omega=\int_L d\omega.\end{equation}\end{lem}

\L ojasiewicz's formula is actually devoted to bounded  subanalytic forms, but the required property is indeed that they are bounded and  extend continuously almost everywhere on the closure of the manifold $L$, which obviously holds true  when the form is $L^\infty$ and $cl(L)$ is $(t;j)$-allowable.

Next we turn to see that integration is well defined for any $(t;j)$ allowable  {\it subanalytic singular simplex}. Let  $\sigma :\Delta_{j} \to X$ be an oriented $(t;j)$-allowable (subanalytic) simplex.  
Denote by $\sigma_{reg}$ the set of points in
$\sigma^{-1}(X_{reg})$  near which $\sigma$ induces a smooth mapping.  Observe that the complement of $\sigma_{reg}$ in $\Delta_j$ has Lebesgue measure zero. Hence, it makes sense to set for  $\omega \in
\Omega_\infty^{j}(X_{reg})$:\begin{equation}\label{eq_def_integrale_simplex}
\int_{\sigma}\omega:=\int_{\Delta_j} \sigma^*\omega=\int_{\sigma_{reg}}  \sigma^*\omega.
\end{equation}

Stokes' formula continues to hold for subanalytic singular $(t;j)$-allowable simplices:
\begin{lem}
Let  $\sigma \in I^tC_j(X)$ and $\omega \in
\Omega_\infty^{j-1}(X_{reg})$.  Then the integral defined in (\ref{eq_def_integrale_simplex}) is finite and:
\begin{equation}\label{eq_stokes_lojasiewicz}
\int_\sigma d\omega =\int_{\partial \sigma} \omega.
\end{equation}
\end{lem}
\begin{proof}
Let $$\Gamma:=\{(x;y) \in X \times \Delta_j : x= \sigma(y)\},$$
and consider a cell decomposition of $\R^{n+j}$ compatible with
$\Gamma$. Refining it, we can assume that the boundary of a cell
is a union of cells.  For simplicity, we will identify
 $\Delta_j$ with $\Gamma$ and assume that $\sigma$ is the canonical projection (restricted to  $\Gamma$).

  Let $C\subset \Gamma$ be a cell of this cell decomposition and let $i:=\dim \,C$. Observe that  either $\sigma_{|C}$ is a diffeomorphism or $\dim \sigma(C)<i$.
   In the former case, if we endow $\sigma(C)$ with the
    orientation induced by  $\Delta_j$ via $\sigma$, we have by definition:
\begin{equation}\label{eq cell et pull back}
\int_C \sigma^*\alpha=\int_{\sigma(C)} \alpha.  \end{equation}
for any $\alpha \in \Omega^i _\infty(X_{reg})$ (we shall need both  the cases $\alpha=\omega$ and $\alpha=d\omega$). If $\dim
\,\sigma(C)<i$, then both vanish and this remains true. Note that, as a cell decomposition is a finite partition, the latter formula already shows that the integral defined in (\ref{eq_def_integrale_simplex}) is finite.

By Lemma \ref{lem lojasiewicz}, if $C$ is a cell of dimension $j$:  \begin{equation}\label{eq_stokes_cell}
\int _{C } \sigma^*d\omega\overset{(\ref{eq cell et pull back})}{=}\int_{\sigma(
C)} d\omega\overset{(\ref{eq stokes l})}{=}\int_{\partial(\sigma(
C))}\omega= \int_{\sigma(\partial
C)}\omega\overset{(\ref{eq cell et pull back})}{=}\int_{\partial C }  \sigma^*\omega ,
\end{equation}
the third equality being true because $\sigma$ is identified with a linear mapping (a projection) on the cell $C$. As $\Delta_j$ is a union of cells,  the latter equality still holds  if we replace $C$ with $\Delta_j$. We conclude that for relevant orientations:
$$\int _{ \sigma } d\omega=\int _{\Delta_j } \sigma^*d\omega\overset{(\ref{eq_stokes_cell})}{=}\int _{\partial \Delta _j} \sigma^*\omega=\int _{\partial \sigma } \omega. $$
\end{proof}

In conclusion, we get that the isomorphism of  Theorem \ref{thm de
rham} is given  by  integrating the differential forms on the  simplices:

\begin{thm}
Let $X$ be a compact  pseudomanifold.  The cochain maps
\begin{eqnarray*}\psi_X ^j:\Omega_\infty ^j(X_{reg}) &\to& I^tC^j (X),\\
\omega  &\mapsto& [\sigma   \mapsto \int_\sigma \omega],
\end{eqnarray*} induce isomorphisms between the cohomology groups.
\end{thm}

To prove it, observe that the cochain map $\psi_X$ induces   a
sheaf homomorphism (recall that the $L^\infty$ forms give rise to
sheaf on the normalization of $X$, see the proof of Theorem \ref{thm de rham}). By the
uniqueness of the map between sheaf cohomology theories with
coefficient in sheaves of $\R$-modules, this map must coincide
with the isomorphism (\ref{eq proof isom de rham}).

%

\begin{rems}
Theorem \ref{thm de rham} still holds if $X$ is a
pseudomanifold  with boundary \cite{gm} (indeed, our Poincar\'e Lemma for $L^\infty$ cohomology does not assume that $X$ is a pseudomanifold). The relative version is then true
as well, by the five lemma. Again, the isomorphism is provided by
integration of forms on allowable chains.

It is worthy of notice that the arguments of the proof of Theorem
\ref{thm de rham} also apply in the noncompact case, establishing
an isomorphism between the cohomology of the {\it locally} bounded forms (locally in $X$, not in $X_{reg}$) and
intersection homology in the maximal perversity.

The results of this paper remain true if we replace the subanalytic category with a polynomially bounded o-minimal structure \cite{driesmiller}. We need the structure to be polynomially bounded for we made use of the preparation theorem for proving Proposition \ref{lem function eq  aux distances}. It is unclear (but not impossible) whether the results of this paper, especially Theorem \ref{thm_retraction}, are valid on a non polynomially bounded o-minimal structure, especially for the $\log-exp$ sets, on which a generalized preparation theorem holds \cite{lr}.
\end{rems}

\section{Appendix I: globally subanalytic sets}
\subsection{Basic definitions}Let $N$ be an analytic manifold.
Recall that a subset $E\subset N$ is called {\bf semianalytic} if it is locally
defined by finitely many real analytic equations and inequalities. More precisely, for each $p \in   N$, there is
a neighborhood $U$ of $p$ in $N$, and real analytic  functions $f_i, g_{ij}$ on $U$, where $i = 1, \dots, r,\, j = 1, \dots , s$, such that
$$E \cap   U = \bigcup _{i=1}^r\bigcap _{j=1} ^s \{x \in U : g_{ij}(x) > 0 \mbox{ and } f_i(x) = 0\}.$$

 A subset $E\subset N$  is {\bf subanalytic}  if it
 can be locally represented as the projection of a semianalytic set. More precisely, for every $p \in  N$,
there exist a neighborhood $U$ of $p$ in $N$, an analytic manifold $P$, and a relatively compact semianalytic
set $Z \subset   N\times  P$ such that $E \cap U = \pi(Z)$, where $\pi :    N \times  P\to    N $ is the natural projection. In particular,
semianalytic sets are subanalytic.

 A subset of $\R^n$ is  {\bf globally subanalytic} if it coincides with a subanalytic subset of $\Pp^n$ after identifying $\R^n$ with and open subset of $\Pp^n$ via:
 $$(y_1, \dots, y_n) \to  (1 : y_1 :\dots: y_n) : \R^n  \to \Pp ^n.$$
We will denote by $\St_n$ the set of  globally subanalytic subsets of $\R^n$.

Clearly,  a bounded subset is subanalytic if and only if it is globally subanalytic.   We say that {\bf a  function is globally subanalytic} if its graph is globally subanalytic.

\subsection{Basic properties of globally subanalytic sets}\label{sect_basic_prop}
Any real algebraic set is globally subanalytic. Furthermore, globally subanalytic sets have the following very useful properties (see \cite{driesmiller}):
 \begin{enumerate}
 \item $\St_n$  is stable under unions, intersections and complement.
 \item If $A \in \St_m$ and $B\in \St_n$ then $A \times B \in \St_{m+n}$.
 \item If $\pi : \R^{n+1} ÔøΩ\to  \R^n$ is the projection on the first $n$ coordinates and $A\in  \St_{n+1}$, then $\pi (A)\in \St_n$.
 \item The elements of $\St_1$ are precisely the finite unions of points and intervals.
\end{enumerate}

  When  a family of sets has these properties, we say that it constitutes an {\it o-minimal structure}. These properties are indeed all the basic properties we need to do most of geometric constructions (such as triangulations, stratifications, retracts, ...). Property $(3)$ is the motivation for introducing subanalytic sets: semianalytic sets are not stable under projection and thus do not fulfill $(3)$.

  Property $(4)$ is a finiteness assumption which makes it possible to derive all the finiteness properties of globally subanalytic sets. The first one and the most important is existence of cell decompositions (see definition below). Most of   the results we give below are not really proper to globally subanalytic sets and are shared by all the sets definable in an o-minimal structure. We will therefore often refer  to \cite{coste} for proofs.

  \begin{dfn}\label{dfn_cell_decomposition}
A {\bf cell decomposition of $\R^n$} is a finite partition of $\R^n$ into globally subanalytic sets
$(C_i)_{i \in I}$, called {\bf cells}, satisfying certain properties explained below.

$n = 1:$ A {\bf cell decomposition of $\R$} is given by a finite subdivision $a_1 < \dots < a_l$ of
$\R$. The cells of $\R$ are the singletons $\{a_i\}$, $0 < i \leq l$, and the intervals $(a_i,
a_{i+1})$,
$0 \leq  i \leq l$, where $a_0 = -\infty$ and $a_{l+1} = +\infty$.

$n > 1:$ A {\bf cell decomposition of $\R^n$} is the data of a cell decomposition of $\R^{n-1}$ and,
for each cell D of $\R^{n-1}$, some globally subanalytic functions  analytic  on $D$ (which is an analytic manifold):
$$\zeta_{D,1} < ... < \zeta_{D,l(D)} : D \to \R.$$

The {\bf cells of $\R^n$} are the {\bf graphs}
$$\{(x, \zeta_{D,i}(x)) : x \in D\} , 0 < i \leq  l(D) ,$$
and the {\bf bands}
$$(\zeta_{D,i}, \zeta_{D,i+1}) := \{(x, y) : x \in D
\mbox{ and }\zeta_{D,i}(x) <y<\zeta_{D,i+1}(x)\},$$
for $0 \leq  i \leq l(D)$, where $\zeta_{D,0} = -\infty$ and $\zeta_{D,l(D)+1} = +\infty$.

A cell decomposition is said to be {\bf compatible with finitely many sets $A_1,\dots,A_k$} if the
$A_i$'s are unions of cells.
\end{dfn}

Given some globally subanalytic sets  $A_1,\dots,A_k$, it is always possible to find a cell decomposition compatible with this family of sets.  Detailed proofs may be found in \cite{coste}.

 This already describes very precisely the geometry of globally subanalytic sets.
 Below, we list some related extra basic properties, useful for us.

{\setlength{\leftmargini}{3pt}
\begin{itemize}
\item {\it (curve selection Lemma)} Let $A\in \St_n$ and  $b\in  cl(A)$. There is an analytic  map $\gamma : (-1, 1) \to  \R^n$
such that $\gamma(0) = b$ and $\gamma((0, 1)) \subset  A$.

\item ({\it subanalycity of the connected components})  Subanalytic sets have only finitely many connected components. They are subanalytic.

\item ({\it uniform bound}) Let  $f : A \to B$ be a globally subanalytic map, with finite fibers. There is $k\in \N$ such that   $card \, f^{-1}(b)\leq k$ for any $b$.
\item ({\it subanalytic choice})  Any globally subanalytic map $f : A \to B$ (not necessarily continuous) admits a globally subanalytic
section, i.e., a globally subanalytic  map $s : B \to  A$ such that $f(s(b))=b $.
\item ({\it subanalycity of the regular locus}) Let $X \in \St_n$. Then $X_{reg}$ is a finite union of analytic manifolds; it is globally subanalytic and dense in $X$.
\end{itemize} }

  For a proof of curve selection Lemma or subanalycity of the regular locus  we refer the reader to \cite{bm}.  A proof of all the other statements may be found in \cite{coste}.

The set $X_{reg}$ is the union of finitely many analytic manifolds. The {\bf dimension of $X$}, denoted $\dim X$, is the  maximal dimension of these manifolds.

\section{Appendix II: Some basic model theoretic principles}
\subsection{Formulae} 
We shall need some basic facts of model theory. We first define what we call  $\la$-formulae. Basically, it is a sequence constituted by  quantifier and some symbols, like for instance $\forall x , \; \exists y, \;\, x \leq 2 yz$. More precisely,  {\bf $\la$-formulae} are defined inductively as follows:
\begin{enumerate}
\item If $f$ is a globally subanalytic function then $ f(x) >0$ and $f(x) =0$ are $\la$-formulae.
\item If $\Phi(x_1,\dots , x_n)$ and $\Psi(x_1,\dots , x_n)$ are  $\la$-formulae, then "$\Phi$  and
$\Psi$", "$\Phi$ or $\Psi$",  and "not  Ω$\Phi$", are $\la$-formulae as well.
\item If $\Phi(y, x)$ is an $\la$-formula, then  $\exists x, \; \Phi(y,x)$
and $\forall x , \; \Phi(y,x)$ are  $\la$-formulae.
\end{enumerate}

  The parameters $x=(x_1,\dots,x_n)$ in $\Phi(x)$ denote the free variables (those which are not quantified).
The symbol $\la$ stands for language:
$\la$-formulae are sentences  "in the language of subanalytic geometry".  Roughly speaking, $\la$-formulae are all the mathematical sentences that one can write using quantifiers, globally subanalytic functions, equalities and inequalities.

As an example, consider  the formula $\Phi(x)$:
 $$\forall \ep >0, \exists \alpha >0,\, \forall y, |x-y| \leq \alpha \Rightarrow |f(x)-f(y)| <\ep.$$
However, the formula $\exists n, \,  n \in \N \mbox{ and }  y =n x$ is {\it not} an $\la$-formula ($\N$ may not be described by an $\la$-formula).  Observe that this formula does not define a subanalytic set of $\R^2$.
 Indeed, the following fundamental result relates $\la$-formulae to subanalytic sets:
\begin{pro}\label{pro_formula}
If  $\Phi(x)$ is an $\la$-formula, then the set $\{x \in \R^n : \Phi(x)\}$ is globally subanalytic.
\end{pro}
To briefly account for this proposition, let us point out that the sentence $\exists y , \, f(x,y)=0$ defines the projection of the set defined by  $f(x,y)=0$. Thus, for this sentence the result follows from Property $(3)$ of section \ref{sect_basic_prop}.  The proposition is then showed by induction on the number of quantifiers (the case of the universal quantifier may be reduced to the existential by considering the negation of the sentence, for more details see \cite{coste}, Theorem $1.13$).

This proposition shows how important it is to work with a category of sets  which is
stable under projection.
As a consequence of this proposition,  the interior and closure of a globally subanalytic set are globally subanalytic.

Another consequence of this proposition is that if the graph of a function is defined by an $\la$-formula, then this function is  globally subanalytic.  It enables to establish that a function is subanalytic without much work.

For instance, if $\xi:\R^n \times [0,1] \to [0,1]$ is a globally subanalytic function then the function defined by  $\zeta(x):=\inf_{t \in [0,1]}\xi(x,t)$ is globally subanalytic. It is then easy to check that if $A$ denotes a subanalytic set then the function $x \mapsto d(x,A)$, which assigns to $x$ the Euclidean distance from $x$ to $A$, is globally subanalytic.



\subsection{Field extensions} Consider all the one variable  globally subanalytic functions which are defined in a right-hand-side
 neighborhood of the origin and identify any two of them which coincide in a small right-hand-side neighborhood of the origin.
  It follows from Puiseux Lemma \cite{pawlucki} that this set is indeed the field of real Puiseux series
  $\sum_{i \geq m} a_i T^{\frac{i}{k}}$, $m \in \Z$, $a_i \in \R$, with $\sum_{i\in \N} a_it^i$ convergent for $t$ in a  neighborhood of zero.
   We shall denote this field $k(0_+)$. We may imbed $\R \hookrightarrow \ko$, sending every real number onto the corresponding constant series.

 We can order this field by setting $f\leq g$ in $k(0_+)$ if  $f(t)\leq g(t)$ for $t$ in a right-hand-side neighborhood of the origin. Observe that the indeterminate $T$ is smaller than any positive real number in $\ko$. Consequently, this field is not Archimedean.

We may also define the {\bf Euclidean  norm} on $k(0_+)^n$, by $|x|:=\sum_{i=1}^n x_i^2 \in k(0_+)$. This  gives rise to a topology.
 A good reference for all the results of this section is \cite{coste}.

\subsection*{Extension of sets and functions} As the composite of globally subanalytic mappings is globally subanalytic, any globally subanalytic function $\xi:\R^n \to \R$ may be extended to a function $\xi_{k(0_+)} :k(0_+)^n \to k(0_+)$,  $\xi_{k(0_+)}(x(T)):=\xi(x(T))$, $x(T)\in \ko$.
Any $\la$-formula may then be also "extended" to $k(0_+)$. For instance the formula $\Phi$, $\exists x \in \R^n , f(x)\geq 0$, where $f$ is an analytic function, has the following extension $\Phi_{k(0_+)}$ to $k(0_+)$:
 $$\exists  x \in k(0_+)^n,  \, f_{k(0_+)}(x) \geq 0.$$ In other words, we can extend a formula by extending the functions this formula involves.

It is not very difficult to  check that if two formulae define the same set in $\R^n$ then their respective extensions define the same set in $\ko^n$ (see \cite{coste}). Hence, we may define the {\bf extension of the set}
$A:=   \{x \in \R^n :\Phi(x)\}$ by setting $$A_{k(0_+)}:=\{x \in k(0_+)^n : \Phi_{k(0_+)} (x)\}.$$
In other words, the extension of a set is obtained by regarding the associated equations and inequalities in the field of real analytic Puiseux series. This set is merely the set of germs of Puiseux arcs lying on $A$.  For instance, the extension of the sphere $S^{n-1}$ (generally still denoted $S^{n-1}$) is the set: $$\{x \in \ko^n: \sum_{i=1} ^n x_i^2=1\}.$$



\subsection*{Generic fibers} Let now $A\subset \R \times \R^n$ be a globally subanalytic set. Regarding the first  variable as a parameter, we will consider this set as a family. We define the {\bf generic fiber} of  this family of sets as (recall that $T\in k(0_+)$ stands for the indeterminate):
$$ A_{0_+}:= \{ x \in k(0_+)^n : (T,x) \in A_{k(0_+)} \}.$$
 It is nothing but the set of germs of Puiseux arcs $x(t)$ such that  $(t,x(t))\in A$ for every $t$ positive small enough. Observe that we can also define the {\bf generic fiber of a family of globally subanalytic functions} $f:A \to \R$, which is the function $f_{0_+}:A_{0_+} \to k(0_+)$ which assigns to $x \in A_{k(0_+)}$ the value $f_{k(0_+)}(T,x)$.

We then can define the subanalytic sets (resp. mappings) of $k(0_+)^n$ as the collection of all the generic fibers of subanalytic families of sets (resp. mappings). This constitutes a family of Boolean algebras which enjoys the same properties as $(\St_n)_{n\in \N}$. Indeed, as they satisfy $(1-4)$ of section \ref{sect_basic_prop}, they then satisfy all  the other properties which come down from these properties.

As a matter of fact, Proposition \ref{pro_formula} is still true if we replace $\R$ with $\ko$. For example, the function $x \mapsto d(x,A)$ is well defined and globally subanalytic if so is $A\subset \ko^n$.


 We can also define {\bf the generic fiber of a formula}. If $\Phi(t,x)$ is a formula, with $x$ and $t$ free variables ($t$ considered as a parameter) we define the generic fiber of $\Phi(t,x)$  as the formula obtained by replacing $t$ with the indeterminate $T$, i.e. we set $\Phi_{0_+}(x):=\Phi_{k(0_+)}(T,x)$. This of course reduces the number of free variables.

\subsection*{Transfer principle}  The study of the generic fiber of a family $A \subset \R \times \R^n$ can provide us information on what happens on the fiber $A_t$ for generic parameter $t$. More precisely, we have the following very important fact. Let $\Phi(t)$ be an $\la$-formula. The formula  $\Phi_{0_+}$ holds true in $k(0_+)$  if and only if $\Phi(t)$ holds for every positive real  number $t$ small enough. This is a consequence of a more general theorem sometimes referred as {\it \L o\'s's Theorem}.

 To make it more concrete, let us illustrate it by some examples. One may easily derive from this fact that if the generic fiber of the family of functions $f:\R \times \R^n \to \R$ is bounded by $1$ then there is $\ep >0$ such that $\sup_{x \in \R^n}f(t,x)$ is not greater that $1$ for any $t \in (0,\ep)$. This is due to the fact that $\forall x ,\, f(t,x) \leq 1$ is an $\la$-formula.

Less easy, but still true, is the fact that if $f_{0_+}$ is continuous then there is $\ep>0$ such that the restriction $f: (0,\ep)\times \R^n \to \R$  is continuous (the non obvious part is the continuity with respect to the parameter $t\in (0,\ep)$ which is not guaranteed  by the above transfer principle, see \cite{coste} section $5.6$ for details). This points out the interplay between the geometry of globally  subanalytic sets of $k(0_+)^{n}$ and the geometry of families of globally subanalytic sets of $\R \times \R^{n}$, for generic parameters.

 To give another illustration, let us focus on the argument of  the proof of Corollary \ref{lem existence proj reg preservant xn+1}. Applying Theorem \ref{thm proj reg hom pres} provides a bi-Lipschitz globally subanalytic  map $H: k(0_+)^{n-1}\to k(0_+)^{n-1}$.  This homeomorphism is the generic fiber of a family of mappings $h:(0,\ep)\times  \R^{n-1} \to (0,\ep)\times  \R^{n-1}$. As the desired property  (see (\ref{eq_e_n_regulier})) holds for $h_{0_+}$, by \L o\'s's Theorem, it also holds for $h_t$, for $t>0$ small enough (since it may be expressed by an $\la$-formula).

To summarize, we get parameterized versions of theorems just by working with a bigger underlying field. Working with a bigger field often makes no difference as most of the time we simply have to write $k(0_+)$ instead of $\R$.

\end{document}